\documentclass{amsart}[12]
\usepackage{amssymb, amsmath}
\usepackage[foot]{amsaddr}
\usepackage{mathrsfs}
\usepackage{comment}

\numberwithin{equation}{section} 

\def\sideremark#1{\ifvmode\leavevmode\fi\vadjust{\vbox to0pt{\vss
 \hbox to 0pt{\hskip\hsize\hskip1em
 \vbox{\hsize3cm\tiny\raggedright\pretolerance10000
 \noindent #1\hfill}\hss}\vbox to8pt{\vfil}\vss}}}%

\title{The Frobenius theorem for $\mathbb{Z}^n_2$-supermanifolds}

\begin{document}


\author{Tiffany Covolo$^1$, Stephen Kwok $^2$,}
\author{Norbert Poncin $^2$}
\address{$^1$Faculty of Mathematics, National Research University Higher School of Economics, 7 Vavilova Str., 117312 Moscow, Russia}
\address{$^2$Mathematics Research Unit, University of Luxembourg, 6, rue Richard Coudenhove-Kalergi, L-1359, Luxembourg}
\email{tcovolo@hse.ru, stephen.kwok@uni.lu, norbert.poncin@uni.lu}

\maketitle

\begin{abstract}
We continue the development of $\mathbb{Z}^n_2$-supergeometry, a natural generalization of classical ($\mathbb{Z}_2$-graded) supergeometry, by proving the Frobenius theorem for integrable distributions on differentiable $\mathbb{Z}^n_2$-supermanifolds. Both the local and global versions of the theorem are addressed.
\end{abstract}

\theoremstyle{plain} \newtheorem{thm}{Theorem}[section]
\theoremstyle{plain} \newtheorem*{mythm}{Theorem}
\theoremstyle{plain} \newtheorem{lem}[thm]{Lemma}
\theoremstyle{plain} \newtheorem{prop}[thm]{Proposition}
\theoremstyle{plain} \newtheorem{cor}[thm]{Corollary}
\theoremstyle{definition} \newtheorem{defn}{Definition}

\newcommand{\R}{\mathbb{R}}
\newcommand{\C}{\mathbb{C}}
\newcommand{\N}{\mathbb{N}}
\newcommand{\Z}{\mathbb{Z}}
\newcommand{\Zs}{\Z_2} 		
\newcommand{\Zn}{\Zs^n} 

\newcommand{\cO}{{\mathcal{O}}} 
\newcommand{\cJ}{\mathcal{J}}
\newcommand{\cT}{\mathcal{T}}
\newcommand{\cA}{{\mathcal{A}}} 

\newcommand{\Ci}{{\mathcal{C}}^{\infty}} 
\newcommand{\Co}{{\cal C}^{0}}
\newcommand{\Om}{{\Omega}}

\newcommand{\scD}{\mathscr{D}}

\section{Introduction}

Recently, motivated by various physical applications including those of string theory \cite{AFT} and parastatistics \cite{YJ}, there has been work on the foundations of the theory of differentiable $\Zn$-supermanifolds. A Berezin-Leites style ringed space definition of $\Zn$-supermanifold was given in \cite{CGP1}, and a proof of the $\Zn$-graded analogue of the celebrated theorem of Batchelor \cite{Bat79} \cite{Bat80} and Gawedzki \cite{Gaw} in \cite{CGP2}. The basic results of differential calculus on $\Zn$-supermanifolds were proven in \cite{CKP}. The present paper continues this development of the foundations of $\Zn$-supergeometry.

The classical Frobenius theorem is a key tool in the classical theory of differentiable manifolds, allowing one to build a submanifold of a manifold from knowledge of its tangent bundle. It is therefore desirable to extend this theorem to $\Zn$-supermanifolds, The extension for $n=1$ is by now well-known; proofs may be found in e.g. \cite{DM}, \cite{Var}. The goal of the present paper is to formulate and prove the Frobenius theorem for differentiable $\Zn$-supermanifolds.

We now explain the structure of the paper. Section \ref{prelim} is dedicated to the preliminaries required, largely developing some basic linear algebra in the $\Zn$-graded context, as well as the generalization of Nakayama's lemma to this context. The crucial property of $\mathcal{J}$-adic Hausdorff completeness of the structure sheaf of a $\Zn$-supermanifold is required to ensure that this linear algebra is well-behaved. Section \ref{dist} is dedicated to a discussion of distributions on $\Zn$-supermanifolds. Here the $\mathcal{J}$-adic topology on the structure sheaf comes into play. In section \ref{locfrob}, we prove the local version of the Frobenius theorem. Here we have the novel phenomenon, not present in $n = 1$-supergeometry, of vector fields that are even but not of degree zero, and we analyze the local structure of such vector fields and show that there exists a local coordinate system in which they take a standard form. We then briefly discuss the concepts of embedded and immersed submanifold of a $\Zn$-supermanifold. In Section \ref{globfrob}, we prove the global version of the $\Zn$-super Frobenius theorem that states that a $\Zn$-supermanifold is foliated by the integral subsupermanifolds associated to a distribution.

The generalization of the results of the theory of Lie algebras, Lie groups and their representations to the $\Zn$-graded category is largely open; however, a discussion of $\Zn$-graded (``color") Lie algebras may be found in \cite{KS}. The $\Zn$-super Frobenius theorem opens the way to a development of the theory of $\Zn$-super Lie groups and their actions on $\Zn$-supermanifolds, in particular homogeneous $\Zn$-superspaces. It would be interesting to address these in future work.\\

\section{Acknowledgments}

S.K. was supported by the Luxembourgish FNR via AFR grant no. 7718798.

\section{Preliminaries}\label{prelim}

We will need the following basic criterion for the invertibility of a square matrix with entries in a commutative $\Zn$-superalgebra.\\

\begin{lem}\label{invertiblemodJ}

Let $R$ be a $\Zn$-supercommutative ring, $J$ the homogeneous ideal generated by the elements of nonzero degree. Suppose that $R$ is $J$-adically Hausdorff complete. Let $T$ be an $n \times n$ matrix with entries in $R$. Then $T$ is invertible if and only if it is invertible mod $J$.
\end{lem}

\begin{proof}
If $T$ is invertible, then clearly $\overline{T} := T \text{ (mod} \, J)$ is invertible. Conversely, suppose that $\overline{T}$ is invertible, i.e. there is some matrix $S$ such that $ST = I + X$, where $X$ has all entries in $J$. Hence it suffices to show that any matrix of the form $I + X$, where $X$ has all entries in $J$, is invertible.

The $\Zn$-graded associative ring $N := M(n, R)$ of all $n \times n$ matrices with entries in $R$ is an $R$-module by scalar multiplication. As $R$ is $J$-adically Hausdorff complete, the same is true of $N$. (This follows from the fact that $N$ is a free $R$-module of finite rank, see \cite{CKP}).

The matrix $I + X$ has a formal inverse given by the geometric series: $(I + X)^{-1} = \sum_{k=0}^\infty (-1)^kX^k$. As $X^k \in M(n, J^k)$, the partial sums $\sum_{k=0}^n (-1)^k X^k$ form a $J$-adically Cauchy sequence in $N$. By Hausdorff completeness of $N$, the geometric series converges to a unique limit in $N$, which is the inverse of $I + X$.
\end{proof}

\medskip

The following graded version of Nakayama's lemma will be a key tool in the sequel.\\

\begin{lem}[$\Zn$-graded Nakayama's lemma]\label{Nakayama}
Let $A$ be a $\Zn$-supercommutative local ring with maximal homogeneous ideal $\mathfrak{m}$, $E$ a finitely generated module for $A$, regarded as an {\it ungraded} ring. Let $J$ be the ideal generated by the elements of nonzero degree, and suppose $A$ is $J$-adically Hausdorff complete. Then:\\

\begin{enumerate}
\item If $\mathfrak{m}E = E$, then $E = 0$. More generally, if $H$ is a submodule of $E$ such that $E = \mathfrak{m}E + H$, then $E = H$.\\
\item Let $\{v_i\}_{1 \leq i \leq p}$ be a basis for the ungraded $k$-vector space $E/\mathfrak{m}E$, where $k := A/\mathfrak{m}$. Suppose $e_i \in E$ lie above $v_i$. Then the $e_i$ generate the $A$-module $E$. If $E$ is a graded module for $A$ and the $v_i$ homogeneous, then we may choose the $e_i$ to be homogeneous of the same parity as the $v_i$.\\
\item Suppose $E$ is a projective $A$-module, i.e. $E \oplus F = A^N$ for some $A$-module $F$. Then $E$ is free, and the $e_i$ of 2) form a basis of $E$.
\end{enumerate}
\end{lem}

\medskip

\begin{proof}
The proof is similar to that of the corresponding lemma in \cite{Var}. We first remark that if $B$ is a {\it commutative} local ring with maximal ideal $\mathfrak{n}$, then a square matrix $R$ with coefficients in $B$ is invertible if and only if its reduction mod $\mathfrak{n}$ is; indeed, if this is so, $det(R) \notin \mathfrak{n}$, hence is a unit in $B$. Now we prove 1). Let $\{e_i\}$ generate $E$. Since $E = \mathfrak{m}E$, we may write $e_i = \sum_j m_{ij} e_j$ for some $m_{ij} \in \mathfrak{m}$. Letting $L$ be the matrix with entries $\delta_{ij} - m_{ij}$, we have:

\begin{align*}
L \begin{pmatrix}
e_1\\
e_2\\
\vdots \\
e_p
\end{pmatrix}= 0.
\end{align*}\\

To prove 1), it suffices to show $L$ has a left inverse. We will prove that $L$ is actually invertible. To see this, let $B = A/J$. The units of $A$ are precisely the elements of $A \backslash \mathfrak{m}$, so $J \subseteq \mathfrak{m}$ by Lem. 2.1 (applied to $1 \times 1$ matrices), hence we have homomorphisms:

\begin{align*}
A \to B = A/J \to k = A/\mathfrak{m}.
\end{align*}\

Let $L_B$ (resp. $L_k$) be the reduction of $L$ mod $J$ (resp. mod $\mathfrak{m}$). $B$ is a commutative local ring with maximal ideal $\mathfrak{m}/J$, and $L_k$ is the reduction of $L_B$ mod $(\mathfrak{m}/J)$. But $L_k$ is the identity, so the remark above implies $L_B$ is invertible. But by Lem. \ref{invertiblemodJ}, this implies $L$ is invertible. For the more general statement, suppose $E = H + \mathfrak{m}E$. Then $E/H = \mathfrak{m}(E/H)$, so $E/H = 0$ by what we proved. 

To prove part 2), we set $H$ equal to the submodule generated by the $e_i$. Then $E = \mathfrak{m}E + H$, whence $E = H$ by what we proved above.

In order to prove part 3), first note that $F$ is finitely generated, and that $k = A^N/\mathfrak{m}^N = E/\mathfrak{m}E \oplus F/\mathfrak{m}F$. Let $(w_j)$ be a basis of $F/\mathfrak{m}F$ and let $f_j$ be elements of $F$ lying over the $w_j$. By 2), the $e_i, f_j$ generate $A^N$, and the $e_i$ (resp. $f_j$) generate $E$ (resp. $F$). Let $X$ denote the $N \times N$ matrix whose columns are the coordinate vectors of $e_1, \dotsc, f_1, \dotsc$ in the standard basis of $A^N$. Then as the $e_i, f_j$ form a basis, we have $XY = I$ for some $N$ by $N$ matrix $Y$ with entries in $A$. Reducing mod $J$, we have $X_B Y_B = I$. But as $B$ is commutative, we have $Y_B X_B = I$ as well. By Prop. \ref{invertiblemodJ}, $X$ has a left inverse over $A$, which must be $Y$. Suppose there is a linear relation between the $e_i$ and $f_j$, and let $x$ be the column vector whose components are the coefficients of this relation. Then $Xx = 0$, but then $x = YXx = Xx = 0$. Hence $E$ is a free module with basis $e_i$.
\end{proof}

\medskip

\noindent {\bf Remark.} It would be interesting to see if it is possible to remove the assumption that the ring $A$ is $J$-adically Hausdorff complete in the statement of the graded Nakayama's lemma. This is not necessary for our purposes here, as we work exclusively with the local rings of germs of functions on a $\Zn$-supermanifold, which are indeed $\mathcal{J}$-adically Hausdorff complete. However, for other applications (e.g. the development of algebraic $\Zn$-supergeometry), one might potentially have to consider rings for which this hypothesis is not satisfied, and it would be desirable to remove it.

\bigskip

\section{Distributions on $\Zn$-supermanifolds}\label{dist}

\begin{defn}
Let $M$ be a $\Zn$-supermanifold. A {\it distribution} on $M$ is a graded subsheaf $\scD$ of the tangent sheaf $\mathcal{T}M$ which is locally a direct factor, i.e. for any point $m \in |M|$ there exists an open neighborhood $U \ni m$ and an $\mathcal{O}(U)$-module $\scD'$ such that $\mathcal{T}_nM = \scD_n \oplus \scD'_n$ for all $n \in U$. We consider $\scD$ as a sheaf of topological modules, endowed with the $\mathcal{J}$-adic topology from $\mathcal{O}$.\\
\end{defn}

The following will allow us to define the key concept of the {\it rank} of a distribution.\\

\begin{lem}\label{linindep}
Let $m \in M$ and let $\{X_i, \chi_\rho\}$ be vector fields defined in a neighborhood of $m$ such that their associated tangent vectors at $m$ are linearly independent in $T_mM$. Then their germs $\{[X_i]_m, [\chi_\rho]_m\}$ are $\mathcal{O}_m$-linearly independent in $[\mathcal{T}M]_m$. 
\end{lem}

\begin{proof}
Choose vector fields $Y_1, \dotsc, Y_l$ such that the tangent vectors $\{(X_i)_m, (\chi_\rho)_m\} \cup \{(Y_1)_m, \dotsc (Y_l)_m\}$ form a basis of $T_mM$. By the $\Zn$-super Nakayama's lemma, the germs $\{[Y_1]_m, \dotsc, [Y_l]_m\} \cup \{[X_i]_m, [\chi_\rho]_m\}$ at $m$ form an $\mathcal{O}_m$-basis of $[\mathcal{T}M]_m$. In particular, $\{[X_i]_m, \dotsc, [\chi_\rho]_m\}$ are $\mathcal{O}_m$-linearly independent.
\end{proof}

\medskip

\begin{prop}\label{rank}
Let $\scD$ be a distribution. Then $\scD$ is a locally free sheaf of $\mathcal{O}_M$-modules. Furthermore, if $|M|$ is connected, $\text{dim}(\scD_m/\mathfrak{m}_m\scD_m)$ is independent of $m$.
\end{prop}

\begin{proof}
Let $m \in M$. Then the super Nakayama's lemma (Lem. \ref{Nakayama}) implies that $\scD_m$ is a free $\mathcal{O}_m$-module, hence that $\scD(U)$ is a free $\mathcal{O}(U)$-module in a neighborhood $U$ of $m$. As $\scD$ is a distribution, there exists $D \subseteq \mathcal{T}M(U)$ such that $[\mathcal{T} M]_n = [\scD]_n \oplus [D]_n$ for all $n \in U$. 

Let $\{X_i, \chi_\rho\}$ be vector fields defined on $U$ such that their associated tangent vectors at $m$ are a homogeneous basis for $\scD_m/\mathfrak{m}\scD_m \subseteq T_mM$, and let $Y_j, \Theta_\sigma$ be vector fields defined on $U$ such that their associated tangent vectors at $m$ are a homogeneous basis for $\scD_m/\mathfrak{m}\scD_m$. The associated tangent vectors of $\{X_i, \chi_\rho, Y_j, \Theta_\sigma\}$ at $n$ form a homogeneous basis of $T_nM$ for any $n \in U$. By Lem \ref{linindep}, $\{X_i, \chi_\rho, Y_j, \Theta_\sigma\}$ form an $\mathcal{O}_n$-basis of $[\mathcal{T}M]_n$ for every $n$. To see that $\{X_i, \chi_\rho\}$ actually generate $\scD$, set $\scD'_n := \text{span}_{\mathcal{O}_n} \{[X_i]_n, [\chi_\rho]_n\} \subseteq [\scD]_n$, and $D'_n:= \text{span}_{\mathcal{O}_n} \{[Y_j]_n, [\Theta_\sigma]_n\}$. Then $[\scD']_n \subseteq [\scD]_n$ and $[D]'_n \subseteq [D]_n$, but as $[\scD']_n \oplus [D]'_n = [\mathcal{T}M]_n = [\mathcal{D}]_n \oplus [D]_n$, the generation is clear.

It follows from this that $\text{dim}(\scD_m/\mathfrak{m}\scD_m)$ is independent from $m \in M$, if $M$ is connected.
\end{proof}

\medskip

\begin{defn}
Let $\scD$ be a distribution on $M$. The {\it rank} of $\scD$ is the graded dimension 

\begin{align*}
\text{rk}(\scD) := \text{dim}(\scD_m/\mathfrak{m}_m\scD_m),
\end{align*}\

\noindent where $m$ is any point of $M$.
\end{defn}

Prop \ref{rank} shows that the rank is well-defined. For a distribution $\mathscr{D}$ on $M$ and a point $m \in |M|$, the real $\Zn$-super vector space $\mathscr{D}_m/\mathfrak{m}_m \mathscr{D}_m$ will often be denoted by $D_m$; geometrically, this is the fiber of the distribution at the point $m$. \\

Since the tangent sheaf is a sheaf of topological modules  in $\Zn$-supergeometry, the topological properties of distributions come into play. As is well-known from commutative algebra, if $R$ is a topological ring endowed with the $I$-adic topology defined by some ideal $I$, and $N$ a submodule of an $R$-module $M$, the $I$-adic topology on $N$ does not agree in general with the subspace topology induced on $N$ by the $I$-adic topology of $M$. However, in our case it is easy to see that the two coincide:

\begin{prop}
Let $\scD$ be a distribution. Then $\scD$, endowed with the $\mathcal{J}$-adic topology, is a topological subsheaf of $\mathcal{T}M$, i.e. the inclusion morphism of sheaves of $\mathcal{O}$-modules $\scD \hookrightarrow \mathcal{T}M$ is a topological embedding. Furthermore, $\scD$ is $\mathcal{J}$-adically Hausdorff complete.
\end{prop}

\begin{proof}
Let $m$ be any point of $M$. As $\scD$ is a distribution, there exists an open subset $U \ni m$ and a submodule $\scD'$ of $\mathcal{T}M(U)$ such that $\mathcal{T}_mM = \scD_m \oplus \scD'_m$. This implies $\mathcal{J}^k_m \cdot \mathcal{T}_mM = \mathcal{J}^k_m \cdot \scD_m \oplus \mathcal{J}^k_m \cdot \scD'_m$. From this, it is easy to see that $(\mathcal{J}^k_m \cdot \mathcal{T}_mM) \cap \scD_m = \mathcal{J}^k_m \cdot \scD_m$. For an arbitrary sheaf of $\mathcal{O}$-modules $\mathcal{F}$, let $\mathcal{J}^k \cdot \mathcal{F}$ denote the sheafification of the presheaf $U \mapsto \mathcal{J}^k(U) \cdot \mathcal{F}(U)$. We have an obvious sheaf morphism $\mathcal{J}^k \cdot \scD \hookrightarrow (\mathcal{J}^k \cdot \mathcal{T}M) \cap \scD$ given by inclusion. Then the preceding discussion implies that $\mathcal{J}^k \cdot \scD = (\mathcal{J}^k \cdot \mathcal{T}M) \cap \scD$, which amounts to saying that the inclusion $\scD \hookrightarrow \mathcal{T}M$ is a topological embedding. That $\scD$ is $\mathcal{J}$-adically Hausdorff complete follows from the fact that it is a locally free sheaf of $\mathcal{O}$-modules and the Hausdorff completeness of $\mathcal{O}$.
\end{proof}

\medskip

We conclude this section by giving the appropriate generalizations of the notions of involutive and integrable distributions to $\Zn$-supergeometry. Recall that the $\Zn$-super Lie bracket on vector fields is:

\begin{align*}
[X, Y](f):= X(Yf) - (-1)^{\langle deg X, deg Y \rangle} Y(Xf).\\
\end{align*}

\begin{defn}
A distribution $\scD$ is {\it involutive} if and only if $\scD$ is closed under $\Zn$-graded Lie bracket, i.e. if $X, Y$ are vector fields in $\scD$ defined near $m$, $[X, Y]$ lies in $\scD$. 
\end{defn}

\medskip

\begin{defn}
A distribution $\scD$ is {\it integrable} if and only if for every point $m \in M$, there exists a coordinate system $(x, \xi)$ centered at $m$ such that $\scD_m$ is spanned by $\{\partial / \partial x^i, \partial / \partial \xi^j\}_{1 \leq i \leq r, j \in {\bf s}}$.
\end{defn}

\bigskip

\section{Local $\Zn$-super Frobenius theorem}\label{locfrob}

We have the following analogue of the classical theorem of Frobenius characterizing integrable distributions.\\

\begin{thm}[Local Frobenius theorem on $\Zn$-supermanifolds]\label{localFrob}
A distribution $\scD$ on a $\Zn$-supermanifold $M$ is integrable if and only if $\scD$ is involutive.
\end{thm}

That an integrable distribution is involutive is obvious, since coordinate vector fields always $\Zn$-supercommute. The converse will occupy the rest of the proof. We begin by proving a series of auxiliary results on the local structure of distributions.\\

\begin{prop}\label{commutingvf}
Let $\scD$ be an involutive distribution. Then $\scD$ has a local basis of supercommuting vector fields in a neighborhood of any point.
\end{prop}

\begin{proof}
Let $m$ be a point of $M$ and $(x, \theta)$ be a coordinate system near $m$ in some open set $U$. Let  $r|{\bf s}$ be the rank of $\scD$. Let $X_i$, $1 \leq i \leq r$ (resp. $\chi_\sigma$, $1 \leq \sigma \leq s$) be degree zero (resp. nonzero-degree) vector fields whose germs at $m$ form a basis of $\scD_m$. The coefficients of the $X_i$ and $\chi_\sigma$ in the basis $\partial/\partial x^j, \partial/\partial \theta^\rho$ form an $r|{\bf s} \times p|{\bf q}$ matrix $T$ having the form 
\begin{align*}
T = 
\begin{pmatrix}
a & \alpha\\
\beta & b
\end{pmatrix}
\end{align*}\

\noindent where $a$ (resp. $b$) is an $r \times p$ (resp. ${\bf s \times q}$) matrix of degree zero functions, and $\alpha$ (resp. $\beta$ ) is an $r \times {\bf q}$ (resp. (${\bf s} \times p$) matrix of nonzero degree functions.  The matrix $T$ must have rank $r|{\bf s}$ by the linear independence of the $X_i$ and $\chi_\sigma$. 

By renumbering coordinates (without regard for their degree), we may assume the left $r|{\bf s}$ square block submatrix $T_0$ of $T$ is an invertible degree zero matrix. These permutations may destroy the standard ordering of the columns of $T$ as well as the coarse $\mathbb{Z}_2$-grading of $T$ by parity; however, note such permutations only switch the columns of $T$ and hence preserve the degrees of the rows of $T$.

As $deg(T_0) = 0$, left multiplication of $T$ by $T_0$ preserves the degree of the rows of $T$. Therefore, we have that $T_0^{-1}T$ is a matrix of the form:

\begin{align*}
\begin{pmatrix}
I_r & 0 & *\\
0 & I_s& *
\end{pmatrix}.
\end{align*}\\

Since $T_0$ is invertible, the span of the rows of $T_0^{-1}T$ is the same as the span of the rows of $T$; the degrees of the rows of $T$ have also been preserved by all operations performed on $T$. We may therefore assume that for some (not necessarily standard) ordering of the $\theta^\sigma$s, the $X_i$ and $\chi_\sigma$ are of the form:

\begin{align}\label{commutatorform}
&X_i = \frac {\partial} {\partial x^i} + \sum_{j> r} a_{ji} \frac {\partial} {\partial x^j} + \sum_{\rho > s} \alpha_{j \rho} \frac {\partial} {\partial \theta^\rho}\\
&\chi_\sigma = \frac {\partial} {\partial \theta^\sigma} + \sum_{i > r} \beta_{ji} \frac {\partial} {\partial x^i} + \sum_{\rho > s} b_{j\rho}  \frac {\partial} {\partial \theta^\rho}
\end{align}\\

By the hypothesis that $\scD$ is involutive, the supercommutators $[X_i, X_j]$ lie in $\scD$, hence $[X_i, X_j] = \sum_{t \leq r} f_t X_t + \sum_{\tau \leq s} \varphi_\tau \chi_\tau$. However, one sees from (\ref{commutatorform}) that $[X_i, X_j]$ is an $\mathcal{O}(U)$-linear combination of the $\partial/\partial x^k$ for $k > r$ and the $\partial/\partial \xi^\rho$ for $\rho > s$. It follows that all $f_t, \varphi_\tau$ are zero, i.e. the $X_j$ all supercommute with each other. The cases of $[X_i, \chi_\sigma]$ and $[\chi_\rho, \chi_\sigma]$ are completely analogous.
\end{proof}

\bigskip

We will now show that given a degree $0$ local vector field which is non-degenerate at a point, there exists a local coordinate system in which the given vector field takes a canonical form. We will need the following simple lemma.

\medskip

\begin{lem}\label{adicconverge}
Let $U$ be an open subset of a $\Zn$-supermanifold, and suppose $\{f_k\}$ converges $\mathcal{J}$-adically in $\mathcal{O}(U)$. Then 

\begin{align*}
[\widetilde{\lim_{k \to \infty} f_k}](x) = \lim_{k \to \infty} \widetilde{f_k}(x).
\end{align*}
\end{lem}

\medskip

\begin{proof} 
Without loss of generality, we may assume $f_k \to 0$ $\mathcal{J}$-adically in $\mathcal{O}(U)$. Then there is some $M$ such that $f_k \in \mathcal{J}$ for all $k \geq M$. Since the reduction of any function in $\mathcal{J}$ is zero, $0 = \widetilde{f}_k$ for all $k \geq M$, which implies the lemma.
\end{proof}

\bigskip

\begin{lem}\label{degzerocan}
Let $X$ be a degree zero vector field defined in a neighborhood of $m$ such that $X_m$ is a nonzero tangent vector. Then there is a coordinate system $(x, \xi)$ centered at $m$ such that $X = \partial / \partial x^1$.
\end{lem}

\begin{proof}
As the result is purely local, we may assume $M$ is some open subset $U$ of $\mathbb{R}^{p|{\bf q}}$, with coordinates $(z, \zeta)$ centered at $m = 0$. We note that this implies $\mathcal{T}M(U)$ is a free $\mathcal{O}(U)$-module. Since $X$ has degree $0$, it leaves the ideal $\mathcal{J}$ invariant and hence induces a vector field on the reduced manifold $|U|$ whose induced tangent vector at $0$ is nonzero. By the classical Frobenius theorem applied to $|U|$, we may assume $X \equiv \partial/ \partial z^1 \text{ (mod } \mathcal{J})$. Thus

\begin{align*}
X = \frac {\partial} {\partial z^1} + \sum_j a_j \, \frac {\partial} {\partial z^j} + \sum_\rho \beta_\rho \, \frac {\partial} {\partial \zeta^\rho}.
\end{align*}\\

\noindent where the $a_j$ and $\beta_\rho$ are in $\mathcal{J}$. The $a_j$ have degree zero, hence must lie in $\mathcal{J}^2$, and there exists a degree zero matrix of functions $(b_{\rho \tau})$ such that $\beta_\rho \equiv \sum_{\rho, \tau} b_{\rho \tau} \zeta^\tau \text{ (mod } \mathcal{J}^2)$. Hence we have

\begin{align*}
X \equiv \frac {\partial} {\partial z^1} + \sum_{\rho, \tau} b_{\rho \tau} \zeta^\rho \frac {\partial} {\partial \zeta^\tau} \hspace{3mm} \text{(mod} \mathcal{J}^2).
\end{align*}\\

We now make a transformation $U \to U$, given by $(z, \zeta) \to (y, \eta)$, where

\begin{align*}
&y = z, \hspace{8mm} \eta = g(z) \zeta,
\end{align*}\

\noindent with $g(z) = (g_{\rho \tau}(z))$ an invertible matrix of degree zero functions to be chosen suitably later. A calculation with the chain rule shows that

\begin{align*}
X \equiv \frac {\partial} {\partial y^1} + \sum_\rho \gamma_\rho \frac {\partial} {\partial \eta^\rho} \text{ (mod } \mathcal{J}^2),
\end{align*}\\

\noindent where $\gamma_\rho = \sum_\tau \zeta^\tau \bigg( \frac {\partial g_{\rho \tau}} {\partial z^1} + \sum_\sigma g_{\rho \sigma} b_{\sigma \tau}\bigg)$. 

Since antiderivatives exist locally, we may choose $g$ to satisfy the matrix differential equation

\begin{align*}
&\frac {\partial g} {\partial z^1} = -gb\\
&g(0) = I.
\end{align*}\\

\noindent Then with such a choice of $g$, $g$ is invertible near $0$, $(y, \eta)$ is a coordinate system near $0$ by the $\Zn$-super inverse function theorem, and 

\begin{align*}
X \equiv \frac {\partial} {\partial y^1} \text{ (mod } \mathcal{J}^2).
\end{align*}\\

So we may as well assume that $X \equiv \partial/\partial y^1 \text{ (mod } \mathcal{J}^2)$ from the beginning. Now suppose that $k \geq 2$, and $X \equiv \partial/\partial y^1_k \text{ (mod } \mathcal{J}^k)$ in some coordinate system $(y_k, \eta_k)$. We will show that if we choose a new coordinate system $(y_{k+1}, \eta_{k+1})$ given by 

\begin{align*}
(y_k, \eta_k) \to (y_{k+1}, \eta_{k+1}), \hspace{3mm} y^i_{k+1} = y^i_k + a^k_i, \eta^\rho_{k+1} = \eta^\rho_k + \beta^k_\rho
\end{align*} \\

\noindent with $a^k_i, \beta^k_\rho$ suitably chosen, then $X \equiv \partial/\partial y^1_{k+1} \text{ (mod } \mathcal{J}^{k+1})$. We have:

\begin{align*}
X = \frac {\partial} {\partial y^1_k} + \sum_j g^k_j \frac {\partial} {\partial y^j_k} + \sum_\rho \gamma^k_\rho \frac {\partial} {\partial \eta^j_k},
\end{align*}\\

\noindent where $g^k_j, \gamma^k_\rho \in \mathcal{J}^k$. Assuming for the moment that $(y_{k+1}, \eta_{k+1})$ is a coordinate system near $0$ (this will be justified later), the chain rule gives

\begin{align*}
\frac {\partial} {\partial y^j_k} = \frac {\partial} {\partial y^j_{k+1}} + \sum_l \bigg(\frac {\partial a^k_l} {\partial y^j_k} \bigg) \frac {\partial} {\partial y^l_{k+1}} + \sum_\tau \bigg(\frac {\partial \beta^k_\tau} {\partial y^j_k} \bigg) \frac {\partial} {\partial \eta^\tau_{k+1}}
\end{align*}\\

\noindent and

\begin{align*}
\frac {\partial} {\partial \eta^\rho_k} = \frac {\partial} {\partial \eta^\rho_{k+1}} + \sum_l \bigg(\frac {\partial a^k_l} {\partial \eta^\rho_k} \bigg) \frac {\partial} {\partial y^l_{k+1}} + \sum_\tau \bigg(\frac {\partial \beta^k_\tau} {\partial \eta^\rho_k} \bigg) \frac {\partial} {\partial \eta^\tau_{k+1}}.
\end{align*}\\

\noindent As $2k - 1 \geq k + 1$, we have the following expression for $X$:

\begin{align*}
X = \frac {\partial} {\partial y^1_{k+1}} + \sum_j \bigg(g^k_j +  \frac {\partial a_j^k} {\partial y^1_k} \bigg) \frac {\partial} {\partial y^j_{k+1}} + \sum_\rho \bigg(\gamma^k_\rho + \frac {\partial \beta^k_\rho} {\partial y^1_k} \bigg) \frac {\partial} {\partial \eta^\rho_{k+1}} + Z,
\end{align*}\\

\noindent where $Z \equiv 0 \, (\text{mod } \mathcal{J}^{k+1})$. If we choose $a_j^k$, $\beta^k_\rho$ such that 

\begin{align*}
&\frac {\partial a_j^k} {\partial y^1_k} = -g^k_j \\
&\frac {\partial \beta^k_\rho} {\partial y^1_k} = -\gamma^k_\rho,
\end{align*}\\

\noindent then $X \equiv \partial/\partial y^1_{k+1} \, (\text{mod } \mathcal{J}^{k+1})$; such a choice is possible, since antiderivatives exist locally. We note that $a_j^k, \beta^k_\rho$ so chosen lie in $\mathcal{J}^k$. This implies the Jacobian matrix of the transformation $(y_k, \eta_k) \mapsto (y_{k+1}, \eta_{k+1})$ at $0$ is the identity, hence $(y_{k+1}, \eta_{k+1})$ is indeed a coordinate system by the inverse function theorem.

Unlike the ordinary super case, $\mathcal{J}$ is not nilpotent, hence this process does not terminate. We obtain instead an infinite sequence of coordinate systems $(y_k, \eta_k)$ such that $X \equiv \partial/\partial y^1_k \, (\text{mod } \mathcal{J}^k)$. We must show that $(y_k, \eta_k)$ converges $\mathcal{J}$-adically to a unique limit $(x, \xi)$, that $(x, \xi)$ is a coordinate system near $0$, and that $X = \partial/\partial x^1$. For any $k$, $y^i_{k+1} - y^i_k$ and $\eta^\rho_{k+1} - \eta^\rho_k$ lie in $\mathcal{J}^k$ by the way in which we have defined them, hence $\{y^i_k\}$ and $\{\eta^\rho_k\}$ are $\mathcal{J}$-adically Cauchy sequences. By Hausdorff completeness of $\mathcal{O}(U)$ \cite{CGP1}, they converge to unique limits $x^i, \xi^\rho$ respectively. 

As the Jacobian matrix of the morphism $(y_2, \eta_2) \to (y_k, \eta_k)$ at $0$ is the identity for each $k$ by construction of $y_k, \eta_k$, the $\mathcal{J}$-adic continuity of vector fields and Lemma \ref{adicconverge} together imply that the Jacobian matrix of the morphism $(y_2, \eta_2) \to (x, \xi)$ at $0$ is also the identity, hence $(x, \xi)$ is a coordinate system near $0$ by the inverse function theorem.

Finally, let $N$ be any positive integer. There exists some $M$ such that $h^k_l := y^l_k - x^l, \omega^k_\tau := \eta_k^\tau - \xi^\tau \in \mathcal{J}^N$ for all $k \geq M$. By the chain rule, we have 

\begin{align*}
\frac {\partial} {\partial y^1_k} = \frac {\partial} {\partial x^1} + \sum_l \bigg(\frac {\partial h^k_l} {\partial y^1_k}\bigg) \frac {\partial} {\partial x^l} + \sum_\tau \bigg(\frac {\partial \omega^k_\tau} {\partial y^1_k}\bigg) \frac {\partial} {\partial \xi^\tau},
\end{align*}\\

\noindent which implies $\partial/\partial y^1_k \equiv \partial/\partial x^1 \, (\text{mod} \, \mathcal{J}^N)$ for all $k \geq M$. Let $N'$ be such that $N' \geq \text{max}(M, N)$. Then $X \equiv \partial/\partial y^1_{N'} \, (\text{mod} \, \mathcal{J}^{N'})$ by the definition of the coordinate systems $(y_k, \eta_k)$ and $\partial/\partial y^1_{N'} \equiv \partial/\partial x^1 \, (\text{mod} \, \mathcal{J}^N)$ from the previous remark. Hence $X \equiv \partial/\partial x^1 \, (\text{mod} \, \mathcal{J}^N)$ since $\mathcal{J}^{N'} \subseteq \mathcal{J}^N$. But $N$ was arbitrary, hence $X = \partial/\partial x^1$ since $\mathcal{T}M(U)$, being a free $\mathcal{O}(U)$-module, is $\mathcal{J}(U)$-adically Hausdorff.
\end{proof}

\medskip

\begin{lem}\label{degzerouptri}
Let $\scD$ be a rank $r|0$ distribution locally generated by a set $\{X_1, \dotsc, X_r\}$ of degree zero $\Zn$-supercommuting vector fields. Then there exists a coordinate system $(x, \xi)$ such that for each $j$, $1 \leq j \leq r$,

\begin{align*}
X_j = \frac {\partial} {\partial x^j} + \sum_{i = 1}^{j-1} a_{ij} \frac {\partial} {\partial x^i}
\end{align*}\

\noindent where the $a_{ij}$ are degree zero functions.
\end{lem}

\begin{proof}
Since the statement is local, we may work from the beginning in a coordinate chart $U$. The proof is by induction on $r$. The case $r = 1$ is Lem. \ref{degzerocan}. Suppose there exist coordinates $(x, \xi)$ on $U$ having the desired property for $r-1$ supercommuting degree zero vector fields; we wish to find a coordinate system having the desired property for $X_1, \dotsc, X_r$. By hypothesis, $X_j = \frac {\partial} {\partial x^j} + \sum_{i = 1}^{j-1} a_{ij} \frac {\partial} {\partial x^i}$ for all $j < r$. We have

\begin{align*}
X_r = \sum_{i=1}^p f_i \frac {\partial} {\partial x^i} + \sum_{\rho = 1}^q \varphi_\rho \frac {\partial} {\partial \xi^\rho}
\end{align*}\\

\noindent where $q$ denotes the number of nonzero-degree coordinates. The hypothesis of supercommutativity implies

\begin{align*}
[X_r, X_j] &= \sum_i f_i \bigg[ \frac {\partial} {\partial x^i}, X_j \bigg] + \sum_\rho \varphi_\rho \bigg[ \frac {\partial} {\partial \xi^\rho}, X_j \bigg] - \sum_i X_jf_i \frac {\partial} {\partial x^i} - \sum_\rho X_j \varphi_\rho \frac {\partial} {\partial \xi^\rho} = 0.
\end{align*}\\

Since $\big[ \frac {\partial} {\partial x^i}, X_j \big]$ are a linear combination of the $\frac {\partial} {\partial x^k}$ for $k < r$, we have that $X_j f_i = 0$ for all $i \geq r$. The upper triangular form of the $X_j$, $j < r$, implies that $f_i$ is a function of $(t^r, \dotsc, t^p, \xi^1, \dotsc, \xi^q)$ alone, for $i \geq r$. Furthermore, $\big[ \frac {\partial} {\partial \xi^\rho}, X_j \big] = 0$ for all $\rho$, hence we see that $X_j \varphi_\rho = 0$ for all $\rho$. By the same reasoning as before, we may conclude that $\varphi_\rho$ is a function only of $(t^r, \dotsc, t^p, \xi^1, \dotsc, \xi^q)$, for all $\rho$. We may thus write:

\begin{align*}
X_r = \sum_{i=1}^{r-1} f_i \frac {\partial} {\partial x^i} + \sum_{k=r}^p f_k \frac {\partial} {\partial x^k} + \sum_{\rho = 1}^q \varphi_\rho \frac {\partial} {\partial \xi^\rho}
\end{align*}\\

Let $X'_r := \sum_{k=r+1}^p f_k \frac {\partial} {\partial x^k} + \sum_{\rho = 1}^q \varphi_\rho \frac {\partial} {\partial \xi^\rho}$. Note that $X'_r$ depends only on the coordinates $x^r, \dotsc, x^p, \xi^1, \dotsc, \xi^q$. Since $X_r$ is $\mathcal{O}$-linearly independent of $X_1, \dotsc, X_{r-1}$, the upper triangular form of the $X_j$ for $j < r$ and Prop. \ref{degzerocan} applied to $X'_r$ imply that we may change the coordinates $x^r, \dotsc, x^p, \xi^1, \dotsc, \xi^q$ so that there is a new coordinate system $(x^1, \dotsc, x^{r-1}, x'^r, \dotsc, x'^p, \xi')$ on $U$ (shrinking $U$ if needed) such that 

\begin{align*}
X_r = \frac {\partial} {\partial x'^r} + \sum_{i=1}^{r-1} f_i \frac {\partial} {\partial x^i},
\end{align*}\

\noindent as desired. It is clear that the upper triangular form of the $X_j$, $j \leq r-1$, remains unchanged when the $X_j$ are expressed in the new coordinate system.
\end{proof}

\medskip

The preceding discussion suffices to prove the local Frobenius theorem for involutive distributions of degree zero. For the general case, we begin by showing the existence of a canonical form for a local, nondegenerate vector field of nonzero degree.\\

\begin{prop}\label{nonzerodegcan}
Let $\chi$ be a vector field on $M$ of nonzero degree defined in a neighborhood of $m$, such that $\chi_m \neq 0$. If $\chi$ is odd, suppose in addition that $\chi^2 = 0$. Then there is a coordinate system $(x, \xi)$ about $m$ such that $\chi = \partial / \partial \xi^\sigma$, where the degree of $\xi^\sigma$ is equal to that of $\chi$.
\end{prop}

\begin{proof}
Since the proposition is local, we may assume that $M = U \subseteq \mathbb{R}^{p|{\bf q}}$ with coordinates $(z, \zeta)$ such that $m$ corresponds to $0$, in which case $\mathcal{T}M(U)$ is a free $\mathcal{O}(U)$-module. Then $\chi$ has the form

\begin{align*}
\chi = \sum_j \alpha_j(z, \zeta) \frac {\partial} {\partial z^j} + \sum_\rho a_\rho(z, \zeta) \frac {\partial} {\partial \zeta^\rho}.
\end{align*}\\

Without loss of generality, the condition $\chi_m \neq 0$ may be taken to be $a_\sigma(0, 0) \neq 0$ for some index $\sigma$ of degree $deg(\chi)$. By reordering the coordinates $\zeta^\rho$ such that $deg(\zeta^\rho) = deg(\chi)$, we may assume that $\zeta^\sigma$ is the first coordinate of degree $deg(\chi)$.

Let $0|{\bf 1}_\sigma$ denote the $\Zn$-superdimension which is $0$ in all degrees except $deg(\sigma)$ and $1$ in degree $deg(\sigma)$, and $p|\widehat{\bf q}$ the codimension of $0|{\bf 1}_\sigma$ in $p|{\bf q}$. Let $\eta^\sigma$ be a coordinate on $\mathbb{R}^{0|{\bf 1_\sigma}}$. We define a morphism $\mathbb{R}^{0|{\bf 1}_\sigma} \times U^{p|\widehat{\bf q}} \to U^{p|{\bf q}}$ by

\begin{align*}
&z^j = y^j + \theta^\sigma \alpha_j(y, 0, \widehat{\eta})\\
& \zeta^\sigma= \eta^\sigma a_\sigma(y, 0, \widehat{\eta})\\
& \zeta^\rho= \eta^\rho + \eta^\sigma a_\rho(y, 0, \widehat{\eta}), \hspace{3mm} \rho \neq \sigma,
\end{align*}\\

\noindent where by convention $\alpha_j(y, 0, \widehat{\eta})$ (resp. $a_\rho(y, 0, \widehat{\eta})$) denotes $\alpha_j(y, \eta)|_{\eta^\sigma = 0}$ (resp. $a_\rho(x, \eta)|_{\eta^\sigma = 0}$). Note that these are functions of $y$ and of $\eta^\rho$ for all $\rho \neq \sigma$. The Jacobian matrix of this morphism at 0 has the form:

\begin{align*}
\begin{pmatrix}
I_{p'} & * & 0 \\
0 & a_\sigma(0) & 0\\
0 & * & I_{q'} 
\end{pmatrix}\\
\end{align*}

\noindent which is invertible, hence the morphism  $(y, \eta) \mapsto (z, \zeta)$ may be regarded as a change of coordinates near $0$ by the inverse function theorem. From the chain rule, we have

\begin{align*}
\frac {\partial} {\partial \eta^\sigma} =  &\sum_j \alpha_j(y, 0, \widehat{\eta}) \frac {\partial} {\partial z^j} + \sum_\rho a_\rho(y, 0, \widehat{\eta}) \frac {\partial} {\partial \zeta^\rho}.
\end{align*}\\

We must express the coefficient functions as functions of the new coordinates $z^j, \zeta^\rho$. For this, note that by the Taylor series expansion of $\Zn$-superfunctions,

\begin{align*}
\alpha_j(z, \zeta) &= \alpha_j(y^i + \eta^\sigma \alpha_i, \eta^\sigma a_\sigma, \eta^{\rho \neq \sigma} + \eta^\sigma a_{\rho \neq \sigma})\\
&= \alpha_j(y, 0, \widehat{\eta}) + \eta^\sigma \kappa 
\end{align*}\

\noindent for some function $\kappa$ of degree $0$. Similarly, we find $a_\rho(z, \zeta) = a_\rho(y, 0, \widehat{\eta})+ \eta^\sigma h$ for some function $h$ of degree $deg(\chi)$. Hence we have

\begin{align*}
\chi = \frac {\partial} {\partial \eta^\sigma} + \eta^\sigma V,
\end{align*}\\

\noindent where $V$ is a vector field of degree zero. The proof now splits into two cases according to the parity of $\chi$.\\

\noindent \underline{\bf $\chi$ even.} The argument mirrors that of Prop \ref{degzerocan}. From the preceding discussion, there exists a coordinate system $(y, \eta)$ such that $\chi \equiv \partial/\partial \eta^\sigma \text{ (mod}\, \mathcal{J})$. Now suppose that $k \geq 1$, and that $\chi  \equiv \partial/\partial \eta^\sigma_k \text{ (mod } \mathcal{J}^k)$ in some coordinate system $(y_k, \eta_k)$. We will show that if we choose a new coordinate system $(y_{k+1}, \eta_{k+1})$ given by 

\begin{align*}
(y_k, \eta_k) \to (y_{k+1}, \eta_{k+1}), \hspace{3mm} y^i_{k+1} = y^i_k + a^k_i, \eta^\rho_{k+1} = \eta^\rho_k + \beta^k_\rho
\end{align*}\\

\noindent with $a^k_i, \beta^k_\rho$ suitably chosen, then $\chi \equiv \partial/\partial \eta^\sigma_{k+1} \text{ (mod } \mathcal{J}^{k+1})$. We have

\begin{align} \label{chiexpress}
\chi \equiv  \frac {\partial} {\partial \eta^\sigma_k} + \sum_j \gamma^k_j \, \frac {\partial} {\partial y^j_k} + \sum_\rho g^k_\rho \, \frac {\partial} {\partial \eta^\rho_k}  
\end{align}\\

\noindent where $\gamma^k_j, g^k_\rho \in \mathcal{J}^k$. Under the assumption that $(y_{k+1}, \eta_{k+1})$ is a coordinate system (which will be justified later), the chain rule again yields

\begin{align*}
\frac {\partial} {\partial y^j_k} = \frac {\partial} {\partial y^j_{k+1}} + \sum_l \bigg(\frac {\partial a^k_l} {\partial y^j_k} \bigg) \frac {\partial} {\partial y^l_{k+1}} + \sum_\tau \bigg(\frac {\partial \beta^k_\tau} {\partial y^j_k} \bigg) \frac {\partial} {\partial \eta^\tau_{k+1}}
\end{align*}\\

\noindent and

\begin{align*}
\frac {\partial} {\partial \eta^\rho_k} = \frac {\partial} {\partial \eta^\rho_{k+1}} + \sum_l \bigg(\frac {\partial a^k_l} {\partial \eta^\rho_k} \bigg) \frac {\partial} {\partial y^l_{k+1}} + \sum_\tau \bigg(\frac {\partial \beta^k_\tau} {\partial \eta^\rho_k} \bigg) \frac {\partial} {\partial \eta^\tau_{k+1}}.
\end{align*}\\

Substituting, we have the following expression for $\chi$:

\begin{align*}
\chi = \frac {\partial} {\partial \eta^\sigma_{k+1}} + \sum_j \bigg( \gamma^k_j + \frac {\partial a^k_j} {\partial \eta^\sigma_k} \bigg) \frac {\partial} {\partial y^j_{k+1}} + \sum_\rho \bigg(g^k_\rho + \frac {\partial \beta^k_\rho} {\partial \eta^\sigma_k} \bigg) \frac {\partial} {\partial \eta^\rho_{k+1}}  + Z,
\end{align*}\\

\noindent where $Z$ denotes the remaining terms (these will be shown to be zero$\text{ mod } \mathcal{J}^{k+1}$ once $a^k_j, \beta^k_\rho$ are suitably chosen). We now choose $a^k_j, \beta^k_\rho$ such that:

\begin{align*}
&\frac {\partial a^k_j} {\partial \eta^\sigma_k} = -\gamma^k_j \text{  and } \beta^k_j(y, 0, \widehat{\eta}) =  0;\\
&\frac {\partial \beta^k_\rho} {\partial \eta^\sigma_k} = -g^k_\rho \text{  and } a^k_\rho(y, 0, \widehat{\eta}) = 0,\\
\end{align*}

\noindent where the notational convention regarding the meaning of $\beta^k_j(y, 0, \widehat{\theta}), a^k_\rho(y, 0, \widehat{\theta})$ continues to hold. Note that the latter conditions ensure that $a^k_j$ and $\beta^k_\rho$ lie in $\mathcal{J}^{k+1}$. Hence one sees that for such a choice of $a^k_j$ and $\beta^k_\rho$, $Z \equiv 0 \text{ (mod } \mathcal{J}^{k+1})$ and so $\chi \equiv \partial / \partial \theta^\sigma_{k+1} \text{ (mod } \mathcal{J}^{k+1})$ as well. Choosing $a^k_j$ and $\beta^k_\rho$ satisfying these conditions is clearly possible if the $\gamma^k_j$ (resp. $g^k_\rho$) are polynomial in $\theta_k^\sigma$, and by Hausdorff completeness of $\mathcal{O}$ and $\mathcal{J}$-adic continuity of vector fields it is possible for arbitrary $\gamma^k_j$ (resp. $g^k_\rho$). The Jacobian matrix of the transformation $(y_k, \eta_k) \mapsto (y_{k+1}, \eta_{k+1})$ at $0$ is clearly the identity and so $(y_{k+1}, \eta_{k+1})$ so defined is a coordinate system by the inverse function theorem, as previously asserted.

As in Prop. \ref{degzerocan}, we obtain an infinite sequence of coordinate systems $(y_k, \eta_k)$ such that $\chi \equiv \partial/\partial \eta^1_k \, (\text{mod } \mathcal{J}^k)$. For any $k$, $\{y^i_k\}$ and $\{\eta^\rho_k\}$ are $\mathcal{J}$-adically Cauchy sequences by construction, hence converge to unique limits $x^i, \xi^\rho$ respectively.  The Jacobian matrix of the morphism $(y_1, \eta_1) \to (y_k, \eta_k)$ at $0$ is the identity for any $k$ by construction of $y_k, \eta_k$, so the Jacobian matrix of the morphism $(y_1, \eta_1) \to (x, \xi)$ at $0$ is also the identity and $(x, \xi)$ is a coordinate system near $0$.

Finally, let $N$ be any positive integer. There exists some $M$ such that $h^k_l := y^l_k - x^l, \omega^k_\tau := \eta^\tau_k - \xi^\tau \in \mathcal{J}^{N+1}$ for all $k \geq M$. By the chain rule, we have 

\begin{align*}
\frac {\partial} {\partial \eta^\sigma_k} = \frac {\partial} {\partial \xi^\sigma} + \sum_l \bigg(\frac {\partial h^k_l} {\partial \eta^\sigma_k}\bigg) \frac {\partial} {\partial x^l} + \sum_\tau \bigg(\frac {\partial \omega^k_l} {\partial \eta^\sigma_k}\bigg) \frac {\partial} {\partial \xi^\tau},
\end{align*}\\

\noindent which implies $\partial/\partial \eta^\sigma_k \equiv \partial/\partial \xi^\sigma \, (\text{mod} \, \mathcal{J}^N)$ for all $k \geq M$. Let $N'$ be such that $N' \geq max(M, N)$. Then $\chi \equiv \partial/\partial \eta^\sigma_{N'} \, (\text{mod} \, \mathcal{J}^{N'})$ by the definition of the coordinate systems $(x_k, \theta_k)$ and $\partial/\partial \eta^\sigma_{N'} \equiv \partial/\partial \xi^\sigma \, (\text{mod} \, \mathcal{J}^N)$ from the previous remark. Hence $\chi \equiv \partial/\partial \xi^\sigma \, (\text{mod} \, \mathcal{J}^N)$. But $N$ was arbitrary, hence $\chi = \partial/\partial \xi^\sigma$ since $\mathcal{T}M(U)$ is $\mathcal{J}(U)$-adically Hausdorff.\\

\noindent \underline{\bf $\chi$ odd.} In this case $(\eta^\sigma)^2 = 0$, and so we may conclude as in the $\mathbb{Z}_2$-graded case by using the hypothesis that $\chi^2 = 0$. Namely, the equation $\chi =\partial / \partial \eta^\sigma + \eta^\sigma V$ then implies

\begin{align*}
\chi^2 &= V-\eta^\sigma W = 0
\end{align*}\

\noindent for some vector field $W$ of degree $\eta^\sigma$, which implies $\chi = \partial / \partial \eta^\sigma$.
\end{proof}

\medskip

Now, we may finally prove the more difficult half of the local Frobenius theorem:\\

\begin{thm}
Let $\scD$ be an involutive distribution of rank $r|{\bf s}$ on a $\Zn$-supermanifold $M$. Then for any point $m$ of $M$, there exist local coordinates $(x, \xi)$ about $m$ such that 

\begin{align*}
\frac {\partial} {\partial x^1}, \dotsc, \frac {\partial} {\partial x^r}, \frac {\partial} {\partial \xi^1}, \dotsc, \frac {\partial} {\partial \xi^s}
\end{align*}\

\noindent form a local $\mathcal{O}$-module basis of $\scD$. 
\end{thm}

\begin{proof}
Let $\{X_1, \dotsc, X_r, \chi_1, \dotsc, \chi_s\}$ be a local basis of $\scD$. By Prop. \ref{commutingvf}, we may assume that these vector fields $\Zn$-supercommute. Hence, $\scD' := \text{span} \{X_1, \dotsc, X_r\}$ is an integral subdistribution of rank $r|0$ and by Lem. \ref{degzerouptri}, there exist local coordinates such that $X_i = \frac {\partial} {\partial x^i}$. 

Consider the representation of $\chi_1$ in terms of the coordinate vector fields:

\begin{align*}
\chi_1 = \sum_j \beta^j_1 \frac {\partial} {\partial x^j} + \sum_\rho b^\rho_1 \frac {\partial} {\partial \xi^\rho}.
\end{align*}\

Then from the fact that $[X_i, \chi_1] = 0$ for all $i$, we conclude that the coefficients $\beta^j_1, b^\rho_1$ depend only on $x^{r+1}, \dotsc, x^p, \xi^1, \dotsc, \xi^q$. The representation of $\chi_1$ may contain multiples of the $\frac {\partial} {\partial x^j}$, $j = 1, \dotsc, r$. However, we may eliminate these terms by subtracting off suitable linear combinations of $X_1, \dotsc, X_r$. 

We may therefore assume that $\chi_1$ depends only on the coordinates $x^{r+1}, \dotsc, x^p, \xi^1, \dotsc, \xi^q$. We note that if $\chi_1$ is odd, $\chi_1^2 = 0$ by the supercommutativity of the vector fields. Hence we may change the coordinates $x^{r+1}, \dotsc, x^p, \xi^1, \dotsc, \xi^q$ by Lem. \ref{nonzerodegcan} in such a way that $\chi_1 = \frac {\partial} {\partial \xi^1}$.

For $\chi_2$, we apply the same idea: the fact that $[\chi_2, X_i] = 0$ and $[\chi_2, \chi_1] = 0$ implies that $\chi_2$ depends only on $x^{r+1}, \dotsc, x^p, \xi^2, \dotsc, \xi^q$ and we may again apply Lem. \ref{nonzerodegcan} to find a new coordinate system in which $\chi_2 = \frac {\partial} {\partial \xi^2}$. We can clearly apply this argument successively to $\chi_3, \dotsc, \chi_s$ to conclude the proof.
\end{proof}

\bigskip

\section{Submanifolds}

The following definitions are straightforward extensions of the $\Zs$-case. As in the ungraded case, one must distinguish between immersed and embedded submanifolds. Note that we omit prefixes such as ``$\Zn$-super-," etc. in order to make the terminology less cumbersome.

\begin{defn}
Let $M$ be a $\Zn$-supermanifold. An {\it immersed submanifold} of $M$ is a pair $(N, j)$ where $N$ is a $\Zn$-supermanifold and $j: N \to M$ is an injective immersion; that is, $|j|: |N| \to |M|$ is injective and $(dj)_m$ is injective at each point $m$ of $M$. An {\it embedding} is a morphism $j: N \to M$ such that $j$ is an immersion and $|j|$ is a homeomorphism of the topological space $|N|$ onto its image in $|M|$.

A pair $(N, j)$ is an {\it embedded submanifold} of $M$ if $j$ is an embedding. $(N, j)$ is a {\it closed embedded submanifold} if $j$ is an embedding and $|j|(|N|)$ is a closed subset of $|M|$.
\end{defn}

We will just need the following simple property of immersions for our proof of the global Frobenius theorem.

\begin{prop}
Let $j: M \to N$ be an immersion. Then $j$ is locally an embedding on $M$: for any point $m \in |M|$, there is an open neighborhood $U \ni m$ such that $j|_U$ is an embedding. 
\end{prop}

\begin{proof}
This is a consequence of the discussion of local forms of morphisms in \cite{CKP}: since $j$ is an immersion, at $m \in |M|$ there exist charts $U$ around $m$ and $V$ around $j(m)$ such that $j$ is a linear embedding.
\end{proof}

\bigskip

\section{Global $\Zn$-super Frobenius theorem}\label{globfrob}

We now turn our attention to the global version of the $\Zn$-super Frobenius theorem.\\

\begin{defn}
Let $\scD$ be a distribution on $M$. An {\it integral submanifold} of $M$ is an immersed submanifold $j: N \to M$ such that for each $n \in N$, $(dj_n)(T_nN) = D_{j(n)}$. An integral submanifold $j: N \to M$ is {\it maximal} if $|N|$ is connected and $N$ contains every other connected integral submanifold of $M$ which has a point in common with it.
\end{defn}

\medskip

We will say is a coordinate system $(x, \xi)$ on a domain $U$ is {\it adapted} to $\scD$ if $\{\partial/\partial x^i, \partial/\partial \xi^j\}$, $1 \leq i \leq r, 1 \leq j \leq s$ is a basis for $\scD(U)$ as $\cO(U)$-module.

\bigskip

The following lemma will allow us to construct maximal integral submanifolds.\\

\begin{lem}\label{maxintmf}
Let $\{N_i\}_{i \in I}$ be any nonempty collection of connected integral submanifolds of $\scD$ passing through a point $p$. Then there is a unique structure of $\Zn$-supermanifold $N$ on $|N| := \cup_i |N_i|$ such that $N$ is an connected integral submanifold of $\scD$ for which the inclusion $N_i \to N$ is an open embedding for each $i$.
\end{lem}

\begin{proof}
We define a topology on $\cup_i |N_i|$ by declaring a set $U \subset \cup_{i \in I} |N_i|$ to be open if and only if $U \cap |N_i|$ is open in $|N_i|$ for all $i$. Noting that the $|N_i|$ are integral submanifolds for the reduced distribution $\widetilde{\scD}$ on $|M|$, it follows that each $|N_i|$ is an open subset of $|N|$ and that $|N|$ is Hausdorff and second-countable via the same arguments as in the classical case. 

It remains to endow $|N|$ with the structure of a $\Zn$-supermanifold. For this, let us take the $\Zn$-supermanifold atlas defined by all charts $(T \cap N, \varphi)$, where $T$ is a single slice in an adapted chart $U$ for $\scD$, and $\varphi: T \to \R^{r|{\bf s}}$ is the morphism which is given in the adapted coordinates by linear projection: $\varphi(x^1, \dotsc, x^r, \dotsc, x^p, \xi^1, \dotsc, \xi^s, \dotsc, \xi^q) = (x^1, \dotsc, x^r, \xi^1, \dotsc, \xi^s)$. (Such a chart exists around any point of $N$ by the local Frobenius theorem \ref{localFrob}). Any slice $T$ is a closed embedded submanifold of $U$.

Let $T, T'$ be two such slices in adapted charts $U, U'$. Then $V := U \cap U'$ is also an adapted chart, and $T \cap V, T' \cap V$ are embedded integral submanifolds of $\scD$ in $V$. Suppose that $|T \cap N|$ and $|T' \cap N|$ intersect. Then a standard lemma on the local structure of integral manifolds for a distribution on an ungraded manifold, applied to $\widetilde{\scD}$, implies $|T \cap V|, |T' \cap V|$ are contained in a single slice $\widetilde{S}$ for $\widetilde{\scD}$ inside $|V|$, and since all integral submanifolds for $\widetilde{\scD}$ are equidimensional, the inclusions $|T \cap V| \to \widetilde{S}, |T' \cap V| \to \widetilde{S}$ are open embeddings, whence $T \cap T'$ is an open subsupermanifold of $T$ and $T'$.

As $T, T'$ are embedded, the transition maps $\varphi' \circ \varphi^{-1}$ are isomorphisms, showing that the atlas defines a $\Zn$-supermanifold structure on $|N|$. This implies that the inclusion $j: N \to M$ is an immersion (since $j$ is locally an embedding). Furthermore $dj_n(T_nN) =  D_{j(n)}$ at any point $n$, since this is true on any slice in an adapted chart around $n$, so that $N$ is an integral submanifold. Suppose $N'$ is another $\Zn$-supermanifold structure on $|N|$ such that $N'$ is an integral submanifold. Then $i': N' \to M$ is an immersion, so locally an embedding. Hence the $\Zn$-supermanifold structure on $N'$ must match that of the slices on sufficiently small open subsets of $|N|$. This proves the desired uniqueness. 
\end{proof}

\bigskip

\begin{thm}[Global Frobenius theorem for $\Zn$-supermanifolds]
Let $M$ be a $\Zn$-supermanifold and $\scD$ an involutive distribution on $M$. Then given any point $m$ of $M$, there exists a unique maximal integral subsupermanifold of $\scD$ passing through $m$.
\end{thm}

\begin{proof}
By the local Frobenius theorem, the collection $\{N_i\}$ of all connected integral submanifolds of $\scD$ passing through $m$ is nonempty. By Lem. \ref{maxintmf}, $\{N_i\}$ determines a unique connected integral submanifold $N_m$ passing through $m$. $N_m$ is clearly maximal since any other connected integral submanifold passing through $m$ belongs to $\{N_i\}$, so must be contained in $N_m$. Uniqueness follows for the same reason: let $N'_m$ be another maximal integral submanifold through $m$. Then $N'_m$ belongs to $\{N_i\}$ and hence is contained in $N_m$. But then $N'_m = N_m$ by maximality.
\end{proof}

\end{document}